\documentclass[a4paper]{amsart}
\usepackage{url}
\usepackage{amssymb}

\swapnumbers
\theoremstyle{plain}
\newtheorem{thm}{Theorem}[section]
\newtheorem{lem}[thm]{Lemma}

\theoremstyle{definition}
\newtheorem{dfn}[thm]{Definition}

\newtheorem{fct}[thm]{Fact}

\theoremstyle{remark}

\providecommand{\ccc}{\textup{ccc}}
\providecommand{\card}[1]{\lvert#1\rvert}

\DeclareMathOperator{\dom}{dom}

\DeclareMathOperator{\pred}{pred}
\DeclareMathOperator{\ran}{ran}

\title{Lightface $\Sigma^1_2$-indescribable cardinals }

\author{David Schrittesser}

\thanks{During the prepararion of this article, the author was supported by FWF-Project 16334; Also, I would like to thank everybody at the Centre de Recerca Matem\`atica, Barcelona, for their support.}
\address{Kurt G\"odel Research Center for Mathematical Logic\\
 W\"ahringer Stra\ss e 25\\
 A-1090 Wien\\
 Austria}
\urladdr{\url{http://www.logic.univie.ac.at/~david}}
\email{david@logic.univie.ac.at}

\date{March 2005}

\begin{document}

\begin{abstract}$\Sigma^1_3$-absoluteness for $\ccc$ forcing means that for any $\ccc$ forcing $P$, ${H_{\omega_1}}^V \prec_{\Sigma_2}{H_{\omega_1}}^{V^P}$.
``$\omega_1$ inaccessible to reals'' means that for any real $r$, ${\omega_1}^{L[r]}<\omega_1$.
To measure the exact consistency strength of ``$\Sigma^1_3$-absoluteness for $\ccc$ forcing and $\omega_1$ is inaccessible to reals'', we introduce a weak version of a
weakly compact cardinal, namely, a (lightface) $\Sigma^1_2$-indescribable cardinal; $\kappa$ has this property exactly if it is inaccessible and $H_\kappa \prec_{\Sigma_2} H_{\kappa^+}$.
\end{abstract}

\maketitle

\section{Introduction}

The result presented in this paper contributes to the study of principles of generic absoluteness (see \cite{bagaria_axioms_generic_absoluteness} for a survey of this area).
These principles can be seen as generalizations of the bounded forcing axioms, like $MA$, $BPFA$, $BSPFA$ and $BMM$ (see \cite{bagaria_bfa_generic_absoluteness}).
Here is the general form of such a principle:
\begin{dfn}
Let $W$ be a definable subclass of $V$, $\Phi$ a class of formulas with parameters and $\Gamma$ a class of forcing notions.
$\mathcal{A}(W,\Phi,\Gamma)$ is the statement that for any formula $\phi$ that belongs to $\Phi$ and for any $P \in \Gamma$, $W^{V}\vDash \phi \iff W^{V^P}\vDash \phi$.
\end{dfn}
We denote $\mathcal{A}(H_{\omega_1}, \mathbf{\Sigma_2}, \Gamma)$ by $\Sigma^1_3$-Abs$(\Gamma)$, pronounced ``$\Sigma^1_3$-absoluteness (for $\Gamma$)'', as $\Sigma_2$ formulas over $H_{\omega_1}$ are equivalent to $\Sigma^1_3$ formulas.
In this special case, the consistency strengths for various classes $\Gamma$ are known.
See \cite{generic_absoluteness} and \cite{more_generic_absoluteness} for proofs (in the table, we denote by $\ccc$, $proper$, $semi-proper$ the obvious classes of forcing notions; $set$ denotes the class of all set-sized forcing notions).
\begin{center}
  \begin{tabular}{ c | c | c | c }
    $\Sigma^1_3$-Abs$(\ccc)$ & $\Sigma^1_3$-Abs$(proper)$ & $\Sigma^1_3$-Abs$(semi-proper)$ & $\Sigma^1_3$-Abs$(set)$ \\ \hline
    ZFC &  ZFC & ZFC & reflecting
  \end{tabular}
\end{center}
\begin{dfn}\cite{goldstern_shelah_bpfa}\label{reflecting}
A regular cardinal $\kappa$ is  \textit{reflecting} \textit{iff} $V_\kappa \prec_{\Sigma_2} V$, or, equivalently, \textit{iff} for any regular $\theta$ and any formula $\phi$ with parameters from $V_\kappa$ such that
$H_\theta \vDash \phi$, there is a regular $\gamma < \kappa$ such that
$H_\gamma \vDash \phi$.
\end{dfn}

Forcing axioms are often considered with respect to their interaction with other interesting propositions.
For example, knowing how to construct a model of $MA$, one may ask what is needed to obtain a model of $MA\; \wedge$ ``every projective set of reals is Lebesgue measurable'' or $MA \;\wedge $ ``$\omega_1$ is inaccessible to reals'' (as in \cite{equicon_results}).
\begin{dfn}

We say $\omega_1$ is \textit{inaccessible to reals} \textit{iff} for any real $r$, ${\omega_1}^{L[r]} < \omega_1$.
\end{dfn}

The fact that $\Sigma^1_3$-Abs$(set)$ implies that $\omega_1$ is inaccessible to reals (in fact $\Sigma^1_3$-absoluteness for $\omega_1$-preserving forcing suffices) is pivotal in setting it apart from the weaker axioms, those that are equiconsistent with $ZFC$.
One can show that $\Sigma^1_3$-Abs$(proper)$ together with the assumption that $\omega_1$ is inaccessible to reals has the full strength of a reflecting cardinal (\cite{more_generic_absoluteness}; see also \cite{thesis}).
The question is: how does the additional assumption that $\omega_1$ is inaccessible to reals interact with those forcing axioms that do not directly imply it? The answer is made available in the following table, where the additional hypothesis is indicated by ``$\Omega$''.
We present a proof of the result on the far left \footnote{This material was part of my thesis \cite{thesis} and was sketched during a talk at the Summer Workshop in Fine Structure Theory (SWIFT) in Bonn, July 2003, and content of a talk given in M\"unster in February 2004.
I'd like to thank Ralph Schindler for his kind invitation and hospitality.
}.
For this we introduce (in section \ref{indescribable}) a large cardinal property, ``lightface $\Sigma^1_2$-indescribable'' (denoted by ``lf-$\Sigma^1_2$-id'' in the table).
\begin{center}
  \begin{tabular}{ c | c | c | c }
    $\Sigma^1_3$-Abs$(\ccc)$ & $\Sigma^1_3$-Abs$(proper)$ & $\Sigma^1_3$-Abs$(semi-proper)$ & $\Sigma^1_3$-Abs$(set)$ \\
     $\wedge \;\Omega$ &            $\wedge\; \Omega$ &                $\wedge\; \Omega$ &                $\wedge\; \Omega$         \\
 \hline
     lf-$\Sigma^1_2$-id &  reflecting & reflecting & reflecting
  \end{tabular}
\end{center}

The proof uses (among other things) a coding technique developed in \cite[theorem~C]{equicon_results}, which we adapt to our needs in section \ref{coding}.
 Also, the proof  uses a bit of fine structure theory; as we are making an effort to make this paper very accessible, we will review some of the facts we rely on in the next section.

\section{Notation, Facts, Definitions}\label{facts}

To define the large cardinal property hinted at, we need a bit of second-order logic.
We differentiate second-order from first-order variables or constant symbols by using upper case for the former and lower case for the latter.
We remind the reader that $\Phi(X_0, \hdots X_k)$ is a $\Sigma^1_n$-formula if $\Phi$ starts with a block of existential quantifiers over second-order variables, with $n$ changes of quantifier, followed by an arbitrary number of first-order quantifiers.
$\Pi^1_n$ means negation of $\Sigma^1_n$.
Remember that for a first-order formula $\Phi$ (mentioning some second-order variables $Y_0, \hdots, Y_r, X_0, \hdots, X_r$)
\[
\langle M,  X_0, \hdots, X_k \rangle \vDash \exists Y_0 \hdots Q Y_r \Phi(Y_0, \hdots, Y_r, X_0, \hdots, X_r)\]
(where $Q$ denotes $\exists$ or $\forall$) exactly if 
\[
\begin{array}{c}
\exists Y_0 \in \mathcal{P}(M) \hdots Q Y_r \in \mathcal{P}(M)
\\ \text{ such that }\langle M, X_0, \hdots, X_k, Y_0, \hdots, Y_r \rangle \vDash \Phi(Y_0, \hdots, Y_r, X_0, \hdots, X_r)
\end{array}
\]
For an elaborate definition, see \cite[p.~7f\,]{kanamori:03}.

Let $(\phi_i)_{i \in \omega}$ enumerate all $\Delta_0$ formulas.
Now we define the value of the $i$-th $\Sigma_1$-Skolem function for $L_\alpha$ ($\alpha$ a limit ordinal), denoted by $h^{L_\alpha}_i$: if $L_\alpha \vDash \exists v \exists w \phi_i(v,w,x)$ we say $h^{L_\alpha}_i(x)=y$ just if $(y,z)$ is the $\leq_L$-least pair such that $L_\alpha \vDash \phi_i(y,z,x)$; $h^{L_\alpha}_i(x)=\emptyset$ otherwise.
By the $\Sigma_1$-Skolem hull of $M$ inside $L_\alpha$, denoted by $h^{L_\alpha}_{\Sigma_1}(M)$, we mean the least set containing $M$ and closed under all $h^{L_\alpha}_i$.
In this case, $h^{L_\alpha}_{\Sigma_1}(M)=\bigcup_{i \in \omega} {h^{L_\alpha}_i }[ M ]$.
Most importantly, $\Sigma_1$-Skolem functions are uniformly $\Sigma_1$-definable, i.e.
there is a $\Delta_0$ formula $\Phi$ such that $h^{L_\alpha}_i(x)=y$ if and only if $L_\alpha \vDash \exists z \Phi(i,x,y,z)$.
We will make use of simple facts about Skolem hulls, like:
\begin{fct}\label{skolemlower}
If $\langle L_\alpha, x \rangle$ is isomorphic to $\langle h^{L_{\bar \alpha} }_{\Sigma_1}(M\cup\{\bar x \}),\bar x \rangle$ and $M$ is transitive, then $\langle L_\alpha, x \rangle = \langle h^{L_\alpha}_{\Sigma_1}(M\cup\{x \}),x\rangle$.
\end{fct}
\begin{fct} \label{skolemupper}
If $L_\alpha= h^{L_{\alpha}}_{\Sigma_1}(M)$ and $\sigma: L_\alpha \rightarrow_{\Sigma_1} L_\beta$, then $\ran(\sigma)=  h^{L_{\beta}}_{\Sigma_1}(\sigma[M])$.
\end{fct}
For details, see \cite[II, 6]{devlin}).

\section{Lightface $\Sigma^1_2$-indescribable cardinals}\label{indescribable}

\begin{dfn}\label{thedefinition}
We say that a cardinal $\kappa$ has the $\Sigma^1_2$ reflection property if whenever 
\[
V_\kappa \vDash \exists X \forall Y \Phi(X,Y,p)
\] where $\Phi$ is first-order in the language of set-theory with $X$ and $Y$ as additional predicates, and $p \in V_\kappa$, then there is $\xi < \kappa$ such that $V_\xi \vDash \exists X \forall Y \Phi(X,Y,p)$.
We say $\kappa$ is (lightface) $\Sigma^1_2$-indescribable if in addition $\kappa$ is inaccessible.
\end{dfn}

\begin{fct}\label{def2}
$\kappa$ is lightface $\Sigma^1_2$-indescribable $\iff$ $\kappa$ is inaccessible and $H_\kappa \prec_{\Sigma_2} H_{\kappa^+}$.
\end{fct}

\begin{proof}
First assume indescribability.
Let $H_{\kappa^+}\vDash \exists x \forall y \phi(x,y,p)$, where $p \in H_\kappa$.
Pick a witness $x_0$ in $H_{\kappa^+}$.
For any transitive $M \in H_\kappa$ containing $x_0$ and $p$, we have $M\vDash \forall y \phi(x_0,y,p)$.
So $H_{\kappa^+}\vDash \exists x \forall y \phi(x,y,p)$ is equivalent to a  $\Sigma^1_2$ assertion over $H_\kappa$, and thus is reflected by some $H_\xi$, for inaccessible $\xi < \kappa$.
Thus $H_{\xi^+}\vDash \exists x \forall y \phi(x,y,p)$, and as $\Sigma_2$ formulas are upward absolute for members of the $H$-hierarchy, 
\[
H_{\kappa}\vDash \exists x \forall y \phi(x,y,p).
\]
For the other direction, let $H_\kappa\vDash\exists X \forall Y \phi(X,Y,p)$.
Then $H_{\kappa^+}$ thinks ``there is an ordinal $\theta$ such that $\exists x \subseteq V_\theta \forall y \subseteq V_\theta \langle V_\theta,x,y \rangle \vDash \phi(x,y,p)$''.
As ``$z=V_\theta$'' is $\Pi_1$ in $z$ and $\theta$, this is seen to be a $\Sigma_2$ statement in the parameter $p$ and so holds in $H_\kappa$.

\end{proof}

\begin{fct}\label{fct:indescribable}
\begin{enumerate}
\item If $\kappa$ has the $\Sigma^1_2$ reflection property, it is a limit cardinal and is not equal to $2^\lambda$ for any $\lambda < \kappa$.
\item \label{Mahlo}There is a stationary set of $\Sigma^1_2$-indescribable cardinals below any Mahlo cardinal.
The least Mahlo is not $\Sigma^1_2$-indescribable.
\item Reflecting implies $\Sigma^1_2$-indescribable which in turn implies the existence of many inaccessibles.
\item If $P$ is a partial ordering
of size less than $\kappa$, then forcing with $P$ preserves the $\Sigma^1_2$-indescribability of $\kappa$.
\end{enumerate}
\end{fct}
\begin{proof}
\begin{enumerate}
\item If $\kappa = \lambda^+$, the $\Pi^1_1$ sentence (with $\lambda$ as a parameter) ``there is no function from $\lambda$ onto the ordinals'' holds in $V_\kappa$. But this sentence can't hold in any $V_\xi$ containing $\lambda$, $\xi < \kappa$.

If $\kappa = 2^\delta$, look at the sentence ``there is a bijection between $\mathcal P (\delta)$ and the ordinals''. Argue as above.
\item Let $\kappa$ be Mahlo. Consider the function that assigns to each ordinal $\eta < \kappa$ the least $\xi$ such that: if $\Phi$ is $\Sigma^1_2$ with a parameter from $V_\eta$ and there is $\alpha < \kappa$ such that $V_\alpha\vDash \Phi$, then there is $\alpha < \xi$ such that $V_\alpha\vDash \Phi$.
Any closure point under this function has the $\Sigma^1_2$ reflection property, and by Replacement in $V_\kappa$, the closure points under this function form a $cub$ subset of $\kappa$.
This prooves the first assertion.
To see that the least Mahlo cannot have the reflection property, observe that $\kappa$ being Mahlo is expressible by a $\Pi^1_1$ statement over $V_\kappa$.
\item The first assertion follows from the previous Fact, as $\Sigma_2$ sentences are upward absolute for members of the $H$-hierarchy.
Secondly, being inaccessible is expressible as a $\Pi^1_1$ statement (for example
the power set of any set exists and can be mapped injectively into some ordinal and there is no function with a set as domain but unbounded range).
\item\label{forcing} By Fact \ref{def2}, it suffices to proove $H_\kappa \prec_{\Sigma_2} H_{\kappa^+}$ holds in the extension, assuming it holds in the ground model. Observe that for a partial order $P$ and a regular $\alpha$ such that $P\in H_\alpha$, every element of $(H_\alpha)^{V^P}$  has a $P$-name in $H_\alpha$ (straightforwardly check, constructing names by induction on the rank; in fact, it suffices to assume that $P$ is a subset of $H_\kappa$ and has the $\kappa$-cc).
Recall that for any given $\Delta_0$ formula $\phi$ and for any transitive $ZF^-$-model $M$ such that $P \in M$, the forcing relation for $\phi$ on $P\times (P$-names in $M)$ is uniformly $\Delta_1$-definable over $M$ (i.e. the definition is the same for all such $M$).  
Thus 
$\Vdash_P$``$H_{\kappa^+} \vDash \exists x \forall y \phi(y,y,\dot{p})$'' is equivalent to a statement of the form
\[ 
\exists \dot{x} \in H_{\kappa^+} \quad \forall \dot{y} \in H_{\kappa^+} \quad H_{\kappa^+} \vDash \phi'(\dot{x},\dot{y},\dot{p},P),
\]
where $\phi'$ is $\Delta_1$.
As $H_\kappa \prec_{\Sigma_2} H_{\kappa^+}$, the above holds with $\kappa^+$ replaced by $\kappa$. As $P \in H_\kappa$, $\Vdash_P$``$H_\kappa \vDash \exists x \forall y \Phi(x,y,\dot{p})$''.
\end{enumerate}

\end{proof}

Lightface $\Sigma^1_2$-indescribability does imply a certain fragment of Mahloness, as we observed in \cite{thesis}.
For an application of the following notion see \cite[\S5]{bagaria_solovay_models}.
\begin{dfn}
Let us call a cardinal $\kappa$ $\Sigma_n$-Mahlo (resp. $\Pi_n$-Mahlo) \textit{iff} it is inaccessible and every $cub$ subset $C$ of $\kappa$ with a $\Sigma_n$ (resp.
$\Pi_n$) definition in $H_\kappa$, with parameters, contains an inaccessible cardinal. Lightface $\Sigma_n$-Mahlo (resp. $\Pi_n$-Mahlo) is defined analogously, but without allowing parameters in the definition of $C$.
\end{dfn}

\begin{fct}
If $\kappa$ is lightface $\Sigma^1_2$-indescribable, then it is $\Sigma_2$-Mahlo. 
In fact, $\kappa$ is an inaccessible limit of $\Sigma_2$-Mahlo cardinals. Any lightface $\Pi_2$-Mahlo cardinal is an inaccessible limit of $\Sigma^1_2$-indescribable cardinals.
\end{fct}
\begin{proof}
First, assume $\kappa$ is $\Sigma^1_2$-indescribable.  To proove that $\kappa$ is $\Sigma_2$-Mahlo, let $C$ be a $cub$ subset of $\kappa$ such that $x \in C \iff H_\kappa\vDash \phi(\xi)$, where $\phi(\xi)$ is $\Sigma_2$.
$H_{\kappa^+}\vDash$``there is an inaccessible $\theta$ such that
$H_\theta\vDash\forall \xi \exists \bar\xi>\xi \phi(\bar\xi)$''.
This statement is itself $\Sigma_2$, so it  also holds in  $H_\kappa$.
So there is an inaccessible $\theta < \kappa$ such that
$H_\theta \vDash \forall \xi \exists \bar \xi > \xi \phi(\bar\xi)$.
By upward absoluteness of $\phi(\xi)$ for members of the $H$-hierarchy and closedness of $C$, $\theta \in C$. This completes the proof of the first assertion.
It is straightforward to check that ``$\kappa$ is $\Sigma_2$-Mahlo'' is $\Pi^1_1$ over $V_\kappa$, so since $\kappa$ is $\Sigma^1_2$-indescribable, there are unboundedly many $\Sigma_2$-Mahlo cardinals below $\kappa$.
For the last assertion, assume $\kappa$ is lightface $\Pi_2$-Mahlo.
We follow the proof of Fact \ref{fct:indescribable}, (\ref{Mahlo}).
Check that the $cub$ set mentioned there, consisting of $\xi < \kappa$ having the $\Sigma^1_2$ reflection property, has a $\Pi_2$ definition, without parameters, over $V_\kappa$.
So there are unboundedly many $\Sigma^1_2$-indescribable cardinals below $\kappa$.
\end{proof}

\section{Coding using an Aronszajn-tree}\label{coding}

We fix the following notation: for a tree $T$, we denote by $<_T$ (or $\leq_T$) the tree order, $T_\alpha$ denotes the $\alpha$-th level of $T$ and $T\upharpoonright \alpha$ denotes the subtree of $T$ consisting of all levels of height less than $\alpha$.
By $\pred(t)$ we mean of course $\{ t' \in T | t' <_T t \}$.
The following works for any Aronszajn-tree $T$, that is, a tree of height $\omega_1$ with countable levels and without any cofinal branches (i.e.
linearly ordered sets of type $\omega_1$).
 Aronszajn trees can be ``specialized'' by a $\ccc$ forcing: that is, one adds an order preserving function from the tree into the rationals (a so-called specializing function). This ensures that one cannot add, by further forcing, cofinal branches without at the same time collapsing $\omega_1$.
Applying this forcing to code a subset of $\omega_1$ by a real, \cite{equicon_results} proves that $MA$ together with ``$\omega_1$ is inaccessible to reals'' implies $\omega_1$ is weakly compact in $L$.
We present a slight variation.

\begin{fct}\label{fct:ccc-resh}
Let $S=(s_\alpha)_{\alpha < \omega_1}$ be a sequence of reals.
There is a $\ccc$ forcing $P$ that adds a real $r$ such that in the extension the following holds: whenever $M$ is a transitive model of $ZF^-$ such that
$r \in M$, $ \langle T,\leq_T \rangle \in M$, we have $(s_\alpha)_{\alpha < \omega_1} \in M$.

\end{fct}

\begin{proof}

To achieve this, we iterate the following notion of forcing: fix  $Q_0$, $Q_1$, two disjoint dense sets whose union is all rational numbers.
For any sequence $S=(s_\alpha)_{\alpha < \omega_1}$ consider $P^{S}_T$ consisting of all conditions $f$ such that

\begin{enumerate}
\item $f$ is a function with domain a finite subset of $T \times \omega$.
\item For each $n \in \omega$, the function $t \mapsto f(t,n)$ is a partial order preserving mapping from $(T,<_T)$ into the rationals. 
\item For any $\alpha < \omega_1$, $t$ at the $\alpha$-th level of $T$ and $n \in \omega$, if $(t,n) \in \dom (f)$, then $f(t,n) \in Q_0$ if and only if $n \in s_\alpha$.
\end{enumerate}

\begin{lem}\label{lem:ccc-resh:basic}
Let $F= \bigcup G$, where $G$ is generic.
Then $F$ is a function from $T \times \omega$ into the rationals which is order preserving and continuous at limit nodes of $T$; 
moreover, for any $\alpha < \omega_1$, and any $t \in T_\alpha$, $\{ \; n \in \omega \; \vert \; F(t,n) \in  Q_0 \; \}= s_\alpha$.
\end{lem}

\begin{proof}
Clearly, $D_{(t,n)}:=\{\; p \in  P^S_T \;\vert\; (t,n) \in \dom (p)\;\}$ is dense for any $(t,n) \in T \times \omega$: given a condition $p$, there is an interval of possible values for $p$ at $(t,n)$ (since $p$ has finite domain), so if $t$ is at level $\alpha$ of $T$, we can choose a value from $Q_0$ or $Q_1$, depending on whether $n \in s_\alpha$ or not.
So $F$ is a total, order preserving function on $T \times \omega$, and the ``moreover'' clause holds by definition.
$F$ is continuous as $D_{(t,n),\epsilon }:=\{ \; p \in  P^S_T \;\vert\; \exists t' \in T \; \lvert p(t',n)-p(t,n) \rvert < \epsilon \;\}$ is dense for any $n \in \omega$, $\epsilon > 0$ and $t$ at a limit level of $T$ (again, by the finiteness of the domain of any condition).
\end{proof}

\begin{lem}
$P^S_T$ is $\ccc$.
\end{lem}

\begin{proof}
Assume $(p_\alpha)_{\alpha < \omega_1}$ is an uncountable antichain; then $\{ \dom(p_\alpha) \vert \alpha < \omega_1 \}$ is an uncountable subset of $[ T \times \omega ]^{< \omega}$, so we can apply the delta-systems lemma and assume that for each $\alpha$, $\dom(p_\alpha)=r \cup d_\alpha$, where $r, (d_\alpha)_{\alpha<\omega_1}$ are pairwise disjoint.
Let us also assume that the $d_\alpha$ all have the same cardinality $k$.
There are only countably many possibilities for the values of the $p_\alpha$ on $r$, so we assume that all the conditions agree on $r$.
So for any $\alpha, \alpha' < \omega_1$, there is $t \in d_\alpha$, $t' \in d_{\alpha'}$ and $n \in \omega$ such that $p_\alpha \cup p_{\alpha'}$ is not order preserving on $\{(t,n), (t',n)\}$, whence in particular $t$ and $t'$ are comparable in the tree order.
As any node of the tree has only countably many predecessors in the tree order, by thinning out $(p_\alpha)_{\alpha < \omega_1}$ we can further assume that for all $\alpha < \alpha' < \omega_1$, there are $t \in  d_\alpha$, $t' \in d_{\alpha'}$ such that $t <_T t'$.
Let us now enumerate the $d_\alpha$ as $t^0_\alpha, \hdots, t^{k-1}_\alpha$.
We know that all the conditions in the antichain have comparable nodes in their domain, we will now find a sufficiently coherent subset of conditions to get a branch through $T$.
Enlarge (using Zorn's lemma) the filter of co-initial subsets of $\omega_1$ to an ultrafilter $U$ ($U$ contains only sets of size $\omega_1$, i.e.
$U$ is uniform).
For any $\alpha < \omega_1$, we have
$\{ \; \beta < \omega_1 \;\vert \; \exists i,j \;\; t^i_\alpha <_T t^j_\beta \; \} \in U$.
So by finite additivity of $U$, for each $\alpha$, there are $i,j$ such that $ \{ \; \beta < \omega_1 \;\vert \; \; t^i_\alpha <_T t^j_\beta \; \} \in U$.
Moreover, there is an uncountable set $I$ and $i,j$ such that the above holds for all $\alpha \in I $ and this particular pair $i,j$.
So for any $\alpha, \alpha' \in I$, as elements of $U$ have non-empty (in fact large) intersection, there is $\beta$ such that $t^i_\alpha <_T t^j_\beta$ and $t^i_{\alpha'} <_T t^j_\beta$, so  $t^i_\alpha$ and $t^i_{\alpha'}$ are comparable and $(t^i_\alpha)_{\alpha \in I}$ is an uncountable branch through $T$.

\end{proof}

Now we can prove Fact \ref{fct:ccc-resh}.
We build $P$ as the finite support iteration of $(P_k)_{k \in \omega}$.

Let $s^0_\alpha=s_\alpha$; $P_0$ is the forcing coding this sequence of reals into a specializing function for $T$.
At stage $n$, we have added a specializing function $F_n$; let $s^{n+1}_\alpha$ be a real coding (in some absolute way) $F_n$ restricted to $(T\upharpoonright \alpha+2) \times \omega$.
$P_{n+1}$ is the forcing for coding the sequence $(s^n_\alpha)_{\alpha< \omega_1}$.
 
Let $r$ be a real coding all reals $(s^k_0)_{k \in \omega}$; we check by induction on $\eta \leq \omega_1$ that $r$ has the  property promised in \ref{fct:ccc-resh}: assume that for all $k$, $(s^k_\xi)_{\xi < \eta} \in M$ (this holds by assumption if $\eta=1$).
If $\eta=\zeta+1$, as $s^{k+1}_\zeta$ codes $F_k$ restricted to $T \upharpoonright \zeta +2$, for an arbitrary $t \in T_{\zeta+1}$, $s^k_{\zeta+1}=\{ n | F_k(t,n)\in Q_0 \} \in M$.
For limit $\eta$, using $(s^k_\xi)_{k \in \omega, \xi < \eta}$ we have $F_k$ restricted to $T\upharpoonright \eta$ inside $M$, and therefore, picking an arbitrary $t \in T_\eta$, $n \in s^k_\eta$ exactly if $sup(\{ F(t',n) \vert t'<_T t\})\in Q_0$; so $s^k_\eta \in M$ for all $k$.

\end{proof}

\section{An equiconsistency}

\begin{thm}
``$\Sigma^1_3$-Abs$(\ccc)$ and $\omega_1$ inaccessible to reals'' has the consistency strength of a $\Sigma^1_2$-indescribable.
 
\end{thm}

\begin{proof}
First, observe that it in order to proove that an inaccessible cardinal $\kappa$ has the $\Sigma^1_2$ reflection property, it suffices to proove the seemingly weaker property where we treat \textit{all second-order quantifiers as ranging over sets of ordinals}, rather than over arbitrary subsets of a structure.
For notational reasons, we shall sometimes identify sets (those denoted by $X$, $\bar X$, $X^\star$ etc.) with their characteristic functions, and therefore write ``$X \upharpoonright \xi$'' for ``$X \cap \xi$''.
Let $\kappa$ denote ${\omega_1}^V$, and work in $L$. 
Observe $\kappa$ is inaccessible.
Let $\Phi$ be some first-order formula (with parameter in $L_\kappa$, which we suppress), and let $X^\star \in L$ be some function from $\kappa$ into $2$, such that
\[
\langle L_\kappa, X^\star \rangle \vDash \forall A \Phi(X^\star, A)
\]
We may naturally assume that for all $\xi < \kappa$ there is $A \subseteq \xi$ such that 
\[
\langle L_\xi, X^\star \upharpoonright \xi, A \rangle \vDash \neg \Phi(X^\star \upharpoonright \xi, A)\text{,}
\] 
for otherwise, we are done\footnote{In other words, we may assume $\kappa$ is not weakly compact in $L$, as witnessed by $X^\star$ and $\Phi$.}. 
Varying the well-known construction of an Aronszajn-tree whose height is an inaccessible cardinal which is not weakly compact in $L$, we now define a tree $T$ and its ordering $\leq_T$:

Elements of $T$ are tuples $(\beta,X)$, where $\beta < \kappa$ and $X \in {}^\delta 2$, for some $\delta$, and
\begin{enumerate}
\item $L_\beta = h^{L_\beta}_{\Sigma_1}(\card{X}\cup\{X\})$ (in particular, $X \in L_\beta$)\label{skolemhull}
\item $X \upharpoonright \card{X} = X^\star \upharpoonright \card{X}$
\item for all $\xi \leq \dom(X)$, there is $A \in L_\beta$, a subset of $\xi$, such that 
\[
\langle L_\xi, X \upharpoonright \xi, A \rangle \vDash \neg \Phi( X \upharpoonright \xi, A).
\]
\label{counterex}
\end{enumerate}
Define $(\beta,X)\leq_T (\bar\beta, \bar X) \iff$ $X\leq_L \bar X$ and there is a $\Sigma_1$-elementary embedding $\sigma:L_\beta \rightarrow L_{\bar\beta}$ such that $\sigma(X)=\sigma(\bar X)$ and $\sigma$ is the identity on $\card{X}$.
This can be motivated by observing that branches correspond to a failure of reflection, as will become clear in a moment.

Let's check $\leq_T$ is a tree order.
Clearly, $\leq_T$ is transitive and reflexive.
Also, $\leq_T$ is antisymmetric: assume $(\beta, X) \leq (\beta',X')$ and $(\beta',X') \leq (\beta, X)$.
As $X=X'$, the embedding witnessing $(\beta, X)\leq_T (\beta',X')$ shows $L_\beta$ is isomorphic to $h^{L_{\beta'}}_{\Sigma_1}(\card{X'}\cup\{X'\})$ (by Fact \ref{skolemupper}); but the latter is just $L_{\beta'}$, by item (\ref{skolemhull}) in the definition of $T$.
It remains to check that any two predecessors of a node are comparable: say $(\beta, X)$, $(\beta', X') \leq_T (\bar \beta, \bar X)$, as witnessed by embeddings $\sigma$ and $\sigma'$.
Without loss of generality assume $X \leq_L X'$, whence also $\card{X}\leq \card{X'}$ (if not, since $X'$ is a function on an ordinal, $X' \in L_{\card{X'}^+} \subseteq L_{\card{X}}$, contradiction).
So (once more using Fact \ref{skolemupper}) $\ran(\sigma)=h^{L_{\bar \beta}}_{\Sigma_1}(\card{X}\cup\{ \bar X \}) \subseteq h^{L_{\bar \beta}}_{\Sigma_1}(\card{X'}\cup\{ \bar X \})=\ran(\sigma')$, whence $(\sigma')^{-1} \circ \sigma$ is a well-defined elementary embedding and so  $(\beta, X) \leq_T (\beta', X')$.

We now show $T$ is a $\kappa$-Aronszajn tree.
First observe that for a node $(\bar \beta,\bar X)$ of $T$ and a cardinal $\alpha \leq \bar \beta$, there is \textit{exactly one} $t\leq_T (\beta,X)$ of cardinality $\alpha$.
Existence: look at the transitive collapse $L_\beta$ of $h^{L_{\bar \beta}}_{\Sigma_1}(\alpha\cup\{ \bar X \})$ and let $X$ denote the image of $\bar X$ under the collapsing map (let $\sigma$ denote the inverse of this map).
Then $\card{X}=\alpha$, so $L_\beta = h^{L_{\beta}}_{\Sigma_1}(\card{X}\cup\{ X \})$, by Fact \ref{skolemlower}.
Item (\ref{counterex}) holds for $(\bar \beta, \bar X)$, so by a Skolem hull argument, it also holds for $(\beta,X)$.
So $(\beta, X) \in T$.
If $\alpha < \card{\bar X }$, $X \leq_L \bar X$, and $\sigma$ witnesses $(\beta,X) \leq_T (\bar \beta,\bar X)$.
If $\alpha = \card{\bar X }$, by item (\ref{skolemhull}), $X=\bar X$ and  $\beta=\bar \beta$.
Uniqueness: say $(\beta, X)$, $(\beta', X') \leq_T (\bar \beta, \bar X)$, and $\alpha=\card{\beta}=\card{\beta'}$.
By Fact \ref{skolemupper}, both $\langle L_\beta, X\rangle$ and $\langle L_{\beta'}, X' \rangle$ are isomorphic to $h^{L_{\bar \beta}}_{\Sigma_1}(\alpha\cup\{ \bar X \})$, so they are identical.
As a corollary we obtain that if $(\beta,X) \in T$ and $\card{\beta} = \omega_\alpha$, $(\beta,X)$ has exactly $\alpha$ predecessors in $\leq_T$, i.e. the height of $(\beta,X)$ in T is $\alpha$.
So $T\upharpoonright \alpha \subseteq L_{\omega_\alpha}$ ($T$ has small levels).
$T$ has height at least $\kappa$: Let any $\alpha < \kappa$ be given. Let $X := X^\star \upharpoonright \omega_\alpha$ and let $L_\beta$ be the transitive collapse of $H:=h^{L_{\kappa}}_{\Sigma_1}(\omega_\alpha\cup\{ X \})$.
It is easy to check that $(\beta,X) \in T$ (for item (\ref{counterex}), observe that $\dom(X) = \omega_\alpha \in H$) and we have seen its height is exactly $\alpha$.
To conclude that $T$ is Aronszajn (in $V$), it remains to check:
\begin{lem}
$T$ does not have a branch of order-type $\kappa$ in $V$.
\end{lem}
\begin{proof}
Else, let $(\beta(\alpha), X(\alpha))_{\alpha < \kappa}$ be such a branch.  Let $\sigma^{\bar \alpha}_\alpha: L_{\beta(\alpha) } \rightarrow_{\Sigma_1}L_{\beta(\bar \alpha)}$ be the embedding witnessing $(\beta(\alpha), X(\alpha)) \leq_T (\beta(\bar \alpha), X(\bar \alpha))$.
A straightforward argument involving the $\Sigma_1$-definable Skolem functions shows that for $\alpha < \alpha' < \bar \alpha$, $\sigma^{\bar \alpha}_{\alpha'}\circ \sigma^{\alpha'}_{\alpha}=\sigma^{\bar \alpha}_{\alpha}$.
As $\kappa$ has uncountable cofinality, the direct limit of this chain of models is well-founded and a model of $V=L$, therefore isomorphic to some $L_\delta$.
Each $L_{\beta(\alpha)}$ is $\Sigma_1$-elementarily embeddable into $L_\delta$ via a map that is the identity on $\card{\beta(\alpha)}^L$, and all the $X(\alpha)$ are mapped to one $X_0$ which must therefore end-extend $X^\star$ (in the sense that $X_0 \upharpoonright \kappa = X^\star)$.
So $\delta > \kappa$ (as $X_0 \in L_\delta$).
By elementarity (and condition \ref{counterex} in the definition of $T$), there is $A \in L_\delta$, a subset of $\kappa$, such that $\langle L_\kappa, X_0 \upharpoonright \kappa , A \rangle \vDash \neg \Phi(X_\delta \upharpoonright \kappa ,A)$, contradiction.
\end{proof}

Let's go back to working in $L$ again, for yet a little while.
$T$ is not pruned (there are dying branches and branches that don't split), and $T$ needn't even have unique limit nodes (in the sense that for $t$ and $t'$ at a limit level $T_\lambda$, if $t$ and $t'$ have the same predecessors, then $t=t'$).
The latter shortcoming has to be remedied, and this is accomplished easily by replacing $T$ by $T'$, where $T'\upharpoonright \omega = T \upharpoonright \omega$, $T'_{\alpha+1}=T_{\alpha}$ for any infinite ordinal $\alpha<\kappa$, while for limit ordinals $\lambda$ we set $T'_\lambda=\{\pred(t)  | t \in T_\lambda \}$.
$T'$ carries the obvious order ($t\leq_{T'} t'$ exactly if either $t \subseteq t'$ or $t \in t'$ or $t \subseteq \pred(t')$ or $t \leq_T t'$).

Fix $\delta^\star$ such that $X^\star \in L_{\delta^\star}$.
Pick $E$, a binary relation on $\kappa$, such that
\begin{enumerate}
\item $\langle \kappa, E \rangle \cong \langle L_{\delta^\star}, \in \rangle$, and
\item $X^\star(\xi)=1 \iff (\xi +1)\; E \; \emptyset$.

\end{enumerate}
Define $C:= \{ \xi < \kappa | \xi$ is a cardinal and $\langle L_\xi, X^\star\upharpoonright\xi, E\cap(\xi\times\xi) \rangle \prec \langle L_\kappa, X^\star, E\rangle \}$.
By inaccessibility of $\kappa$ this is a $cub$ set.
Let $C$ be enumerated as $(c_\xi)_{\xi<\kappa}$.

Now we work in $V$:
let $s_\xi$ be a real coding, in some absolute manner, the tuple 
\[(T_{\xi+1},X^\star\upharpoonright c_\xi, E\cap (c_\xi \times c_\xi)).\]
Apply the forcing just described (Fact \ref{fct:ccc-resh}) to code the sequence $S=(s_\xi)_{\xi<\omega_1}$ into a single real $r$, using $T$.

Consider any $\beta < \kappa$ such that $L_\beta[r]$ is a model of ``$ZF^-$ and $\omega_1$ exists''. Let $\alpha$ denote $\omega_1^{L_\beta[r]}$.
We claim that for some $\xi \leq \alpha$, there is $x \in L_\beta \cap \mathcal P (\xi)$ such that for all $a \in L_\beta \cap \mathcal P (\xi)$, $\langle L_\xi, x, a\rangle \vDash \Phi(x,a)$, i.e.
that from the point of view of $L_\beta$, reflection occurs before or at $\omega_1$.
Assume otherwise; we show how to recursively reconstruct $(s_\xi)_{\xi < \alpha}$ inside $L_\beta[r]$, and then obtain a contradiction.
We construct  $(s_\xi)_{\xi < \eta}$ by recursion on $\eta \leq \alpha$.
$s_0 \in L_\beta[r]$ is immediate.
Now say $\eta=\gamma+1$: by induction hypothesis $(s_\xi)_{\xi \leq \gamma} \in L_\beta[r]$, so $T\upharpoonright \gamma+2 \in L_\beta[r]$. 
As in the proof of Fact \ref{fct:ccc-resh}, utilizing the specializing functions on that tree (coded recursively by $r$), we obtain $s_{\gamma+1} \in L_\beta[r]$.

We shall now consider two cases simultaneously, since the next few steps of the argument are identical for both:
\begin{enumerate}
\item \label{omega1}$\eta < \alpha$ is a limit ordinal; in this case, we must show how to continue the construction of $(s_\xi)_{\xi<\alpha}$.
\item \label{recursion}we have constructed $(s_\xi)_{\xi<\alpha}$ and $\eta=\alpha$; this leads to a contradiction.
\end{enumerate}
In any case, we may assume $(s_\xi)_{\xi<\eta} \in L_\beta[r]$, whence $X^\star \upharpoonright c_\eta$, $E \cap (c_\eta \times c_\eta) \in L_\beta[r]$.
$E \cap (c_\eta \times c_\eta)$ is of course a well founded relation, and by the definition of $C$ and elementarity, its transitive collapse is equal to some $L_{\zeta^\star}$ such that $X^\star \upharpoonright c_\eta \in L_{\zeta^\star}$ and $\zeta^\star < \beta$.

To be sure the construction of $(s_\xi)_{\xi < \alpha}$ takes place entirely in $L_\beta[r]$, we feel we should mention the triviality that since $(\omega_\eta)^L < \beta$, $(\omega_\eta)^L=(\omega_\eta)^{L_\beta}$.
Work in $L$.
Since $c_\eta \geq \omega_\eta$, $X^\star \upharpoonright \omega_\eta \in L_\beta$.
For each $\xi < \eta$, look at the transitive collapse $L_{\beta(\xi)}$ of $h^{L_\beta}_{\Sigma_1}(\omega_\xi \cup \{X^\star \upharpoonright \omega_\eta\})$, and let $X(\xi)$ be the image of $X^\star \upharpoonright \omega_\eta$ under the collapsing map.
By definability of the Skolem-hull operator and by Replacement in $L_\beta$, the sequence $(\beta(\xi),X(\xi))_{\xi < \eta}$ is an element of $L_\beta$.
Observe that for $\xi < \eta$, $c_\xi$ is countable in $L_\beta[r]$, so $c_\eta \leq \alpha$ and thus $\omega_\eta \leq \alpha$. 
Hence, by assumption, for each $\xi \leq \omega_\eta$ there is $a \in L_\beta$, $a \subseteq \xi$, such that $\langle L_\xi,X^\star\upharpoonright \xi,  a \rangle\vDash \neg \Phi(X^\star\upharpoonright \xi,a)$.
This ensures item (\ref{counterex}) in the definition of $T$ holds for each $(\beta(\xi), X(\xi))$, so arguing just as in the proof showing $T$ is Aronszajn, we have $(\beta(\xi), X(\xi)) \in T$ and its height is $\xi$.
So we have found a branch $b$ of order-type $\eta$ in $T \upharpoonright \eta$, $b \in L_\beta$.

We work in $V$ again. Once more, as in the proof of Fact \ref{fct:ccc-resh}, we may recover, for all $k$, $F_k \upharpoonright (b\times \omega)$ from $b$ and $r$, inside $L_\beta[r]$ (observe only one node on each level suffices for the construction).
Now we proceed to argue by cases: in case (\ref{omega1}), $b$ has order-type $\alpha = {\omega_1}^{L_\beta[r]}$, contradicting $F_0 \upharpoonright (b\times \omega) \in L_\beta[r]$, as $F_0$ yields an order-preserving function from $b$ into the rationals.
For case (\ref{recursion}), we must show $s_\eta \in L_\beta[r]$, and indeed, this holds as $n \in s_\eta \iff \sup\{ F_0(t,n) | t \in b \} \in Q_0$. This finishes the proof of the claim.

So we have found, after forcing with a $\ccc$ partial order, a real $r$ with the $\Pi_1$ property
\[
\begin{array}{c}
\forall \beta < \omega_1 \text{, if }L_\beta[r]\vDash \text{``}ZF^-\text{ and $\omega_1$ exists'', then}\\
\exists \xi \leq (\omega_1)^{L_\beta[r]}\text{ such that }(L_\xi\vDash \exists X \forall A \Phi(X,A))^{L\beta}
\end{array}
\]
By $\Sigma^1_3$-absoluteness, we may assume that $r$ is in the ground model.
But since $\omega_1$ is inaccessible to reals, we may look at $\beta:= (\omega_2)^{L[r]}$.
By the above, for some $\xi \leq (\omega_1)^{L_\beta[r]}$, $(L_\xi\vDash \exists X \forall A \Phi(X,A))^{L\beta}$, and therefore $(L_\xi\vDash \exists X \forall A \Phi(X,A))^L$.
This completes one direction of the proof.

For the other direction, assume $\kappa$ is $\Sigma^1_2$-indescribable.
We show that after forcing with the L\'evy-collapse of $\kappa$, $\Sigma^1_3$-absoluteness for $\ccc$ forcing holds and $\kappa$ is inaccessible to reals.
The latter is clear, as any real in the extension can be absorbed into an intermediate model where $\kappa$ is still inaccessible.

In $V^{Coll(\omega,  \kappa  )}$, let $P$ be a $\ccc$ partial order which forces a $\Sigma^1_3(r)$ statement $\phi(r)$, $r \in V^{Coll(\omega, \{ \alpha \} )}$, for some $\alpha < \kappa$.

Firstly, we can assume that $\card{P} = \omega_1$: using the tree representation of $\Sigma^1_2$ sets, write $\phi(r)$ as ``there is a real $x$ such that $T(x)$ is well-founded''. Here, $T$ is a tree on $\omega_1$ which is $\Delta_0$ definable in the parameters $r$ and $\omega_1$. 
So $\Vdash_{P} ``\exists \dot{x}$ such that $T(\dot{x})$ is well-founded''.
As $P$ has the $\ccc$, there is $\xi < \omega_2$ such that $\Vdash_{P} rank(T(\dot{x}))<\check{\xi}$, and there is a name $\dot{F}$ for a ranking function on $T(\dot{x})$, $\card{\dot{F}} = \omega_1$. Now let $M$ be an elementary submodel of $H_{\omega_2}$ such that $\dot x, \dot F$ and $ P$ are elements of $M$ and $\omega_1 +1 \subseteq M$. As the forcing relation for $\Delta_0$ sentences is uniformly $\Delta_0$ definable for transitive models of $ZF^-$, we can take the transitive collapse of $M$ and we have $\Vdash_{P'}$``$\dot F'$ is an order preserving function from $T(\dot x')$ into the ordinals'', where $\dot x', \dot F'$ and $P'$ are the images of $\dot x, \dot F$ and $ P$ under the collapsing map. Thus, since $P'$ preserves $\omega_1$, $\Vdash_{P'}\phi(r)$. This proves we can assume $P$ has size $\omega_1$.

For the moment, we work in $ W:=V^{Coll(\omega, \{ \alpha \} )}$.
Let $\dot P$ be a $Coll(\omega, \kappa)$-name for $P$. As $Coll(\omega, \kappa) \in H_{\kappa^+}$, we may assume $\dot P \in H_{\kappa^+}$, 
whence
\begin{equation}\label{forcable}
H_{\kappa^+}\vDash \exists Q \Vdash_Q \phi(r),
\end{equation}
as witnessed by $Q:=Coll(\omega, \kappa ) * P$.
In $W^Q$,  
\[
\phi(r) \iff H_{\kappa^+}\vDash \exists u \forall w \psi  (u,w,r)
\]
for a suitable $\Delta_0$ formula $\psi$ (e.g. such that $\forall w \psi(u,w,r)$ says that a certain tree $T(u,r)$ on $\omega$ is ill-founded, i.e. has no ranking function; use the tree representation for $\Pi^1_1$ sets).
Using this equivalence and arguing as in \ref{fct:indescribable}(\ref{forcing}), in $W$ \eqref{forcable} is equivalent to 
\begin{equation}\label{forcable2}
\exists Q \in H_{\kappa^+} \quad\exists \dot{u} \in H_{\kappa^+} \quad \forall \dot{w} \in H_{\kappa^+} \quad H_{\kappa^+} \vDash \psi'(\dot{u},\dot{w},r,Q),
\end{equation}
where $\psi'$ is $\Delta_1$. By $\Sigma^1_2$-indescribability of $\kappa$, \eqref{forcable2} holds with $\kappa^+$ replaced by $\kappa$. As \eqref{forcable} and \eqref{forcable2} are still equivalent when $\kappa^+$ is replaced by $\kappa$, there is $Q' \in H_\kappa$, $\Vdash_{Q'} \phi(r)$.

In $V^{Coll(\omega, \kappa)}$, there is $H$ which is generic for $Q'$ over $W$, as $\card{\mathcal P (Q')}^W$ is collapsed to $\omega$.
$\phi(r)$ holds in $W[H]$, and $\phi(r)$ is upward absolute between $W[H] \subseteq V^{Coll(\omega, \kappa)}$ and $V^{Coll(\omega, \kappa)}$.
So $\phi(r)$ holds in $V^{Coll(\omega, \kappa)}$, whence $\Sigma^1_3$-absoluteness holds between this model and any subsequent $\ccc$ extension.
\end{proof}

\section*{Open questions}

What are other applications of lightface indescribable cardinals? E.g.
what is the consistency strength of ``two-step'' $\Sigma^1_3$-absoluteness for $\ccc$ forcing  plus $\omega_1$ inaccessible to reals? Two-step $\Sigma^1_3$-absoluteness for $\ccc$ forcing means that for any $\ccc$ forcing $P$ and a $P$-name $\dot Q$ such that $P$ forces ``$\dot Q$ is a $\ccc$ partial ordering'', ${H_{\omega_1}}^V \prec_{\Sigma_2} {H_{\omega_1}}^{V^P}$ and $H_{\omega_1}^{V^P} \prec_{\Sigma_2} {H_{\omega_1}}^{V^{P\star \dot Q}}$.

\bibliography{biblio}
\bibliographystyle{amsalpha}

\end{document}